\documentclass[10pt,draft,twoside]{amsart}
\usepackage{amssymb}
\usepackage{latexsym}
\usepackage{amsfonts}
\usepackage{amsmath}

\DeclareMathOperator{\AGL}{AGL}
\DeclareMathOperator{\Aut}{Aut}

\DeclareMathOperator{\Rep}{Rep}

\DeclareMathOperator{\Diff}{Diff}
\DeclareMathOperator{\PSL}{PSL}
\DeclareMathOperator{\PGaL}{P\Gamma L}

\DeclareMathOperator{\PGL}{PGL}
\DeclareMathOperator{\GaL}{\Gamma L}
\DeclareMathOperator{\Sp}{Sp}

\DeclareMathOperator{\PSU}{PSU}

\DeclareMathOperator{\supp}{supp}
\DeclareMathOperator{\pt}{\mathcal{P}}

\DeclareMathOperator{\B}{\mathcal{B}}

\DeclareMathOperator{\Ree}{Ree}

\DeclareMathOperator{\z}{{\bf{0}}}
\DeclareMathOperator{\1}{{\bf{1}}}
\DeclareMathOperator{\GL}{GL}

\DeclareMathOperator{\wt}{wt}

\DeclareMathOperator{\base}{\mathfrak{B}}
\DeclareMathOperator{\topg}{\mathfrak{L}}

\newtheorem{theorem}{Theorem}[section]
\newtheorem{proposition}[theorem]{Proposition}
\newtheorem{lemma}[theorem]{Lemma}
\newtheorem{corollary}[theorem]{Corollary}

\theoremstyle{definition}

\newtheorem{remark}[theorem]{Remark}
\newtheorem{definition}[theorem]{Definition}
\newtheorem{example}[theorem]{Example}
\newtheorem{hypothesis}[theorem]{Hypothesis}
\newcommand{\sym}[1]{{\sf Sym}\,#1}

\renewcommand{\wr}{\,{\sf wr}\,}

\newcommand{\F}{\mathbb F}

\renewcommand{\leq}{\leqslant}
\renewcommand{\geq}{\geqslant}

\numberwithin{equation}{section}

\begin{document}

\title[Classification of a family of completely transitive codes.]
      {Classification of a family of\\ Completely Transitive Codes}
\author{Neil I. Gillespie, Michael Giudici and Cheryl E. Praeger}
\address{[Gillespie, Giudici and Praeger] Centre for the Mathematics of Symmetry and Computation\\
School of Mathematics and Statistics\\
The University of Western Australia\\
35 Stirling Highway, Crawley\\
Western Australia 6009}

\email{neil.gillespie@graduate.uwa.edu.au, michael.giudici@uwa.edu.au,
  cheryl.praeger@uwa.edu.au}

\begin{abstract}
The \emph{completely regular codes} in Hamming graphs have a high
degree of combinatorial symmetry and have attracted a lot of interest
since their introduction in 1973 by Delsarte.  This paper studies the
subfamily of \emph{completely transitive codes}, those in which an
automorphism group is transitive on each part of the distance
partition.  This family is a natural generalisation of the binary
completely transitive codes introduced by Sol{\'e} in 1990.  We take
the first step towards a classification of these codes, determining
those for which the automorphism group is faithful on entries.  
\end{abstract}

\thanks{{\it Date:} draft typeset \today
{\it Key words and phrases: completely
  transitive codes, completely regular codes, automorphism groups, Hamming graphs.} 
The first author is supported by an Australian Postgraduate Award and
by the Australian Research Council Federation Fellowship FF0776186 of
the third author.}

\maketitle

\section{Introduction}\label{intro}

Completely regular codes have been studied extensively ever since
Delsarte \cite{delsarte} introduced them as a generalisation of
perfect codes in 1973.  Not only are these codes of interest to coding
theorists as they possess a high degree of combinatorial symmetry,
but, due to a result by Brouwer et al. \cite[p.353]{distreg}, they are also
the building blocks of certain types of distance regular graphs.  At
present there is no general classification of completely regular
codes.  However, certain families of completely regular codes have
been characterised.  For example, the first and third authors proved
that certain binary completely regular codes are uniquely determined
by their length and \emph{minimum distance} \cite{hadpap,nord}.
Borges et al. have classified all linear completely regular codes that have 
\emph{covering radius $\rho=2$} and an \emph{antipodal} dual code
\cite{brz1}, showing that these codes are extensions of linear
completely regular codes with covering radius $\rho=1$.  They also
classified this later family of codes, and proved that these codes are in fact  
\emph{coset-completely transitive}, a family of linear completely
regular codes that were first introduced in the binary case by
Sol{\'e} \cite{sole} and then over an arbitrary finite field by the
second and third authors \cite{giupra}.  In \cite{nonexist2}, Borges
et al. classified binary coset-completely transitive codes with
minimum distance at least $9$, showing that the \emph{binary
repetition code} (see Definition \ref{repcode}) is the unique code
in this family.        

In their paper \cite{giupra}, the second and third authors also
generalised coset-completely transitive codes by introducing
\emph{completely transitive codes}, which are defined for not
necessarily linear codes over an arbitrary alphabet, and are also
completely regular.  There exist completely transitive codes that are not
coset-completely transitive.  For example, the first and third authors
proved that certain Hadamard codes \cite{hadpap} and the
Nordstrom-Robinson codes \cite{nord} are completely transitive, but as
they are non-linear, cannot be coset-completely transitive; also the
repetition code of length $3$ over a finite field $\F_q$ for $q\geq
9$ is an example of a linear completely transitive code that is not
coset-completely transitive.  In this paper we begin the
classification of this class of completely regular codes.  To do this
we consider codes as subsets of the vertex set of the Hamming graph
$H(m,q)$, which is a natural setting to study codes of length $m$ over
a finite alphabet $Q$ of size $q$.  The automorphism group of
$\Gamma=H(m,q)$, denoted by $\Aut(\Gamma)$, is isomorphic to $S_q\wr S_m$, and via the
homomorphism given in (\ref{mu}), has an action on the entries of
codewords with kernel $\base\cong S_q^m$.  Given any code $C$ in
$H(m,q)$ with covering radius $\rho$, we define the \emph{distance
  partition of $C$}, $\{C_0=C,C_1,\ldots,C_\rho\}$, which is a
partition of the vertex set of $H(m,q)$ (see Section \ref{cinh}).  If
there exists $X\leq\Aut(\Gamma)$ such that $C_i$ is an $X$-orbit for
$i=0,\ldots,\rho$, we say $C$ is \emph{$X$-completely transitive}, 
or simply \emph{completely transitive}, and we prove the following.  

\begin{theorem}\label{maink=1} Let $C$ be a code in $H(m,q)$ with
  minimum distance $\delta$ such that $|C|\geq 2$ and $\delta\geq
  5$.  Then $C$ is $X$-completely transitive with $X\cap\base=1$ if
  and only if $q=2$, $C$ is equivalent to the binary repetition
  code, and $X\cong S_m$.      
\end{theorem}

In Section \ref{secdef} we introduce the necessary definitions and
preliminary results required for this paper.  For the remainder of the
paper we consider $X$-completely transitive codes with
$X\cap\base=1$.  In particular in Section \ref{secamsm}, we consider
such codes with $X\cong A_m$ or $S_m$, and with $\delta=m$.  We deduce
that for $X$-completely transitive codes with $\delta\geq 5$ and
$X\cap\base=1$, the group $X$ has a $2$-transitive action on entries,
and therefore is of affine or almost simple type, and in Section
\ref{affinetype} and Section \ref{sk=1blkfaith} we consider the
respective cases.  Finally in Section \ref{proof} we prove Theorem
\ref{maink=1}.

\section{Definitions and Preliminaries}\label{secdef}

\subsection{Codes in Hamming Graphs}\label{cinh} 
The \emph{Hamming graph $H(m,q)$} is the graph $\Gamma$ with vertex
set $V(\Gamma)$, the set of $m$-tuples with entries from an
\emph{alphabet} $Q$ of size $q$, and an edge exists between two 
vertices if and only if they differ in precisely one entry.
Throughout we assume that $m, q\geq 2$.  Any code of length $m$ over an
alphabet $Q$ of size $q$ can be embedded as a subset of $V(\Gamma)$.  
The automorphism group of $\Gamma$, which we
denote by $\Aut(\Gamma)$, is the semi-direct product
$\base\rtimes\topg$ where $\base\cong S_q^m$ and $\topg\cong S_m$, see
\cite[Thm. 9.2.1]{distreg}.  Let $g=(g_1,\ldots, g_m)\in \base$,
$\sigma\in \topg$ and $\alpha=(\alpha_1,\ldots,\alpha_m)\in V(\Gamma)$. Then
$g\sigma$ acts on $\alpha$ in the following way:         
\begin{equation}\label{acteqn}
  \alpha^{g\sigma}=(\alpha_{1^{\sigma^{-1}}}^{g_{1^{\sigma^{-1}}}},\ldots,\alpha_{m^{\sigma^{-1}}}^{g_{m^{\sigma^{-1}}}})\end{equation}
Let $M=\{1,\ldots,m\}$, and view $M$ as the set of vertex entries of
$H(m,q)$.  Let $0$ denote a distinguished element of the alphabet
$Q$. For $\alpha\in V(\Gamma)$, the \emph{support of $\alpha$} is the set
$\supp(\alpha)=\{i\in M\,:\,\alpha_i\neq 0\}$.  The \emph{weight of
  $\alpha$} is defined as $\wt(\alpha)=|\supp(\alpha)|$.    For any
$a\in Q\backslash\{0\}$ we use the notation $(a^k,0^{m-k})$ to denote
the vertex in $V(\Gamma)$ that has $a$ in the first $k$ entries, and
$0$ in the remaining entries, and if $k=0$ we denote the vertex by
${\bf{0}}$.  

\begin{lemma}\label{supplem} Let $\alpha={\bf{0}}$ and
  $x=(g_1,\ldots,g_m)\sigma\in\Aut(\Gamma)_\alpha$.  Then
  $\supp(\beta^x)=\supp(\beta)^\sigma$ for all $\beta\in V(\Gamma)$.     
\end{lemma}

\begin{proof}  Since each $g_i$ fixes $0$ and $\sigma$ permutes coordinates, 
the $i$th entry of $\beta^x$ is non-zero if and only if the $i^{\sigma^{-1}}$
entry of $\beta$ is non-zero.  Thus the result follows.  
\end{proof}

\noindent For all pairs of vertices $\alpha,\beta\in V(\Gamma)$, the
\emph{Hamming distance} between $\alpha$ and $\beta$, denoted by
$d(\alpha,\beta)$, is defined to be the number of entries in which the
two vertices differ.  This is equal to the length of the shortest path
in the graph between $\alpha$ and $\beta$.  We let $\Gamma_k(\alpha)$ denote the set of
vertices in $V(\Gamma)$ that are at distance $k$ from $\alpha$.        

Let $C$ be a code in $H(m,q)$.  The \emph{minimum distance $\delta$ of C} 
is the smallest distance between distinct codewords of $C$.
For any vertex $\gamma\in V(\Gamma)$, the \emph{distance of
  $\gamma$ from $C$} is equal to $d(\gamma,C)=\min\{d(\gamma,\beta)\,:\,\beta\in C\}.$  The
\emph{covering radius $\rho$ of $C$} is the maximum distance any vertex in $H(m,q)$ is from $C$.  
We let $C_i$ denote the set of vertices that are distance $i$ from $C$, and deduce,
for $i\leq \lfloor (\delta-1)/2\rfloor$, that 
$C_i$ is the disjoint union of $\Gamma_i(\alpha)$ as $\alpha$ varies 
over $C$.  Furthermore, $\{C=C_0,C_1,\ldots,C_\rho\}$ forms a
partition of $V(\Gamma)$, called the \emph{distance
  partition of $C$}.  The \emph{distance distribution of $C$} is the 
$(m+1)$-tuple $a(C)=(a_0,\ldots,a_m)$
where \begin{equation}\label{ai}a_i=\frac{|\{(\alpha,\beta)\in 
  C\times C\,:\,d(\alpha,\beta)=i\}|}{|C|}.\end{equation}   
We observe that $a_i\geq 0$ for all $i$ and $a_0=1$.  Moreover,
$a_i=0$ for $1\leq i\leq \delta-1$ and $|C|=\sum_{i=0}^ma_i$.  In the
case where $q$ is a prime power, the \emph{MacWilliams transform} of $a(C)$ is the $(m+1)$-tuple 
$a'(C)=(a_0',\ldots,a_m')$
where \begin{equation}\label{kracheqn}a'_k:=\sum_{i=0}^ma_iK_k(i)\end{equation}
with \begin{equation*}K_k(x):=\sum_{j=0}^k(-1)^j\binom{x}{j}\binom{m-x}{k-j}(q-1)^{k-j}.\end{equation*}
It follows from \cite[Lem. 5.3.3]{vanlint} that $a'_k\geq 0$ for $k=0,1,\ldots,m$. 

For $\alpha=(\alpha_i)$, $\beta=(\beta_i)\in V(\Gamma)$, we let
$\Diff(\alpha,\beta)=\{i\in M\,:\,\alpha_i\neq\beta_i\}.$  Now suppose
$|C|\geq 2$ and $\alpha,\beta\in C$.  Then we
let $$\Diff(\alpha,\beta,C)=\{\gamma\in    
C\,:\,\Diff(\alpha,\gamma)=\Diff(\alpha,\beta)\}.$$ 
By definition, $\beta\in\Diff(\alpha,\beta,C)$, so $\Diff(\alpha,\beta,C)\neq\emptyset$. 
\begin{lemma}\label{favvertex} Let $C$ be a code with minimum distance
  $\delta$ and $|C|\geq 2$, and let $\alpha,\beta\in C$ such that
  $d(\alpha,\beta)=\delta$.  Then for all $a\in Q$, there exists
  $x\in\Aut(\Gamma)$ such that the following two conditions hold.
\begin{itemize}
\item[(i)] $\alpha^x=(a,\ldots,a)$, and
\item[(ii)] for each $\gamma\in\Diff(\alpha,\beta,C)$,
  $\gamma^x=(c^{\delta},a^{m-\delta})$ for some $c\in
  Q\backslash\{a\}$.    
\end{itemize}
\end{lemma}

\begin{proof}  Let $\Diff(\alpha,\beta,C)=\{\beta^1,\ldots,\beta^s\}$.
  It follows that $\beta^i|_k=\alpha_k$ for each $i\leq s$ and $k\in
  M\backslash\Diff(\alpha,\beta)$.  Therefore, because $C$ has minimum
  distance $\delta$, $d(\beta^i,\beta^j)=\delta$ for each distinct
  pair $\beta^i,\beta^j\in\Diff(\alpha,\beta,C)$.  This implies that
  for each $k\in\Diff(\alpha,\beta)$, the $s+1$ entries
  $\alpha_k,\beta^1|_k,\ldots,\beta^s|_k$ are pairwise distinct
  elements of $Q$.  Thus $s\leq q-1$.  Let $a\in Q$ and
  $\{c_1,\ldots,c_s\}\subseteq Q\backslash\{a\}$.  Since $S_q$ acts
  $q$-transitively on $Q$, it follows that for each $k\in
  \Diff(\alpha,\beta)$ there exists $h_k\in S_q$ such that
  $(\beta^i|_k)^{h_k}=c_i$ for each $i\leq s$ and $\alpha_k^{h_k}=a$.
  Also for each $k\in M\backslash\Diff(\alpha,\beta)$ let
  $h_k=(a\,\,\alpha_k)\in S_q$. Now let
  $h=(h_1,\ldots,h_m)\in\base$.  Since $S_m$ acts $m$-transitively on
  $M$ and $|\Diff(\alpha,\beta)|=\delta\leq m$, there exists
  $\sigma\in S_m$ such that
  $\Diff(\alpha,\beta)^\sigma=\{1,\ldots,\delta\}$.  Let
  $x=h\sigma\in\Aut(\Gamma)$.  Then $\alpha^x=(a,\ldots,a)$ and
  $(\beta^i)^x=(c_i^\delta,a^{m-\delta})$ for each $i=1,\ldots,s$.        
\end{proof}

We say two codes $C$ and $C'$ in $H(m,q)$ are \emph{equivalent} if there exists $x\in\Aut(\Gamma)$ such that
$C^x=C'$.  If $C=C'$, then $x$ is an \emph{automorphism of $C$}, and the \emph{automorphism group of
  $C$} is the setwise stabiliser of $C$ in $\Aut(\Gamma)$, which we denote by $\Aut(C)$. 

Finally, for a set $\Omega$ and group $G\leq\sym(\Omega)$, we say $G$
acts \emph{$k$-homogeneously on $\Omega$} if $G$ acts transitively on
$\Omega^{\{k\}}$, the \emph{set of $k$-subsets of $\Omega$}. 

\subsection{$s$-Neighbour transitive codes.}

\begin{definition}\label{ctrdef} Let $C$ be a code in $H(m,q)$ with distance
  partition $\{C,C_1,\ldots,C_\rho\}$ and $s$ be an integer with $0\leq s\leq\rho$.  If there exists
  $X\leq\Aut(\Gamma)$ such that $C_i$ is an $X$-orbit for $i=0,\ldots,s$, we say $C$ is \emph{$(X,s)$-neighbour transitive}, 
  or simply \emph{$s$-neighbour transitive}.  We observe that $(X,s)$-neighbour transitive codes
  are necessarily $(X,k)$-neighbour transitive for all $k\leq s$.  Moreover, $X$-completely transitive codes
  (defined in Section \ref{intro}) correspond to $(X,\rho)$-neighbour transitive codes.       
\end{definition}

\begin{remark}\label{remfav} Let $y\in\Aut(\Gamma)$, and let $C$ be an 
  $(X,s)$-neighbour transitive code with minimum
  distance $\delta$.  By following a similar argument to that used in \cite[Sec. 2]{fpa}, 
  it holds that $C^y$ is $(X^y,s)$-neighbour transitive, and because minimum
  distance is preserved by equivalence, $C^y$ has minimum distance $\delta$.  Thus for any $a\in Q\backslash\{0\}$,
  Lemma \ref{favvertex} allows us to replace $C$ with an equivalent $(X,s)$-neighbour
  transitive code with minimum distance $\delta$ that contains ${\bf{0}}$ and $(a^\delta,0^{(m-\delta)})$.          
\end{remark}

Let $X\leq\Aut(\Gamma)$ and consider the following
homomorphism:  \begin{equation}\label{mu}\begin{array}{c c c c}   
     \mu:&X&\longrightarrow&S_m\\ 
&g\sigma&\longmapsto&\sigma\\
\end{array}\end{equation} Then $\mu$ defines an action of $X$ on
$M=\{1,\ldots,m\}$, and the kernel of this action is equal to
  $X\cap\base$.  In this paper we are interested in $X$-completely
transitive codes with $X\cap\base=1$, that is $X$ has a faithful
action on $M$.  Hence, in this case we can identify $X$ with $\mu(X)$.     

\begin{definition}\label{repcode} The \emph{repetition code in $H(m,q)$}, denoted by
  $\Rep(m,q)$, is equal to the set of vertices of the form
  $(a,\ldots,a)$, for all $a\in Q$.  It has minimum distance
  $\delta=m$.
\end{definition}  

\begin{example}\label{exk=1} Let $C=\Rep(m,2)$ and let $\alpha$ be the zero
  codeword.  We show that $C$ is $X$-completely transitive with
  $X\cong S_m$ as in Theorem \ref{maink=1}.  It is clear that
  $\topg\leq\Aut(C)$ and that
  $H=\langle(h,\ldots,h)\rangle\leq\Aut(C)$, where $1\neq h\in S_2$.  
  In fact, $\Aut(C)=\langle H,\topg\rangle\cong H\times\topg$ \cite{fpa}.  Now let $X$ be
  the group consisting of automorphisms of the form
  $x=(h,\ldots,h)\sigma$ if $\sigma$ is an odd permutation and
  $x=\sigma$ if $\sigma$ is an even permutation.  Then $X\cong S_m$,
  $X_\alpha\cong A_m$, $X\cap\base=1$, and $X$ acts transitively on
  $C$.  The covering radius of $C$ is $\lfloor\frac{m}{2}\rfloor$
  and $C_i$ consists 
  of the vertices of weights $i$ and $m-i$, for $i=0,\ldots,\lfloor\frac{m}{2}\rfloor$.  Let $\nu_1,\nu_2\in C_i$.  If 
  $\nu_1,\nu_2$ both have the same weight, then because $A_m$ acts
  $i$-homogeneously on $M$ for all $i\leq m$ it follows that there
  exists $\sigma\in X_\alpha$ such that $\nu_1^\sigma=\nu_2$.  Now
  suppose $\nu_1$ and $\nu_2$ have different weights, say $\nu_1$ has weight
  $i$ and $\nu_2$ has weight $m-i$.  Then there exists $x\in X$ such
  that $\nu_2^x$ has weight $i$.  Consequently there exists $\sigma\in
  X_\alpha$ such that $\nu_1^\sigma=\nu_2^x$, thus $\nu_1^{\sigma
    x^{-1}}=\nu_2$.  Hence $X$ acts transitively on $C_i$ and so $C$
  is $X$-completely transitive.  
\end{example}

\begin{proposition}\label{ihom} Let $C$ be an $(X,s)$-neighbour transitive code with minimum distance $\delta$.  
  Then for $\alpha\in C$ and $i\leq\min(s,\lfloor\frac{\delta-1}{2}\rfloor)$, the stabiliser 
  $X_\alpha$ fixes setwise and acts transitively on $\Gamma_i(\alpha)$.  In particular,
  $X_\alpha$ acts $i$-homogeneously on $M$.    
\end{proposition}

\begin{proof} By replacing $C$ with an equivalent code if necessary, Remark \ref{remfav} allows us to assume 
  that $\alpha={\bf{0}}\in C$.  Firstly, because automorphisms of the Hamming graph preserve distance, it
  follows that $X_\alpha\leq X_{\Gamma_i(\alpha)}$.  Now let $\nu_1,\nu_2\in\Gamma_i(\alpha)$.  
  As $C_i$ is an $X$-orbit, and because $\Gamma_i(\alpha)\subseteq C_i$, there exists $x\in X$ such that $\nu_1^x=\nu_2$.  
  Suppose $x\notin X_\alpha$.  Then $\alpha\neq \alpha^x\in C$, and so
  $d(\alpha,\alpha^x)\geq \delta$.  However, $d(\alpha,\alpha^x)\leq 2i<\delta$, 
  which is a contradiction.  Thus $X_\alpha$ acts transitively on $\Gamma_i(\alpha)$.  

  Finally, let $J_1$, $J_2\in M^{\{i\}}$, and $\nu, \gamma\in V(\Gamma)$ such
  that $\supp(\nu)=J_1$ and $\supp(\gamma)=J_2$.  It follows that
  $\nu,\gamma\in\Gamma_i(\alpha)\subseteq C_i$.  As 
  $X_\alpha$ acts transitively on $\Gamma_i(\alpha)$, there exists
  $x=(g_1,\ldots,g_m)\sigma\in X_\alpha$ such that $\nu^x=\gamma$.  A
  consequence of Lemma \ref{supplem} is that
  $J_1^\sigma=\supp(\nu)^\sigma=\supp(\nu^x)=\supp(\gamma)=J_2$.  
  Hence $X_\alpha$ acts $i$-homogeneously on $M$.   
\end{proof}

\begin{corollary}\label{multtrans} Let $C$ be an $(X,s)$-neighbour transitive code with
  minimum distance $\delta$.  Then for each $i\leq\min(s,\lfloor\frac{\delta-1}{2}\rfloor)$ and
  $I\in M^{\{i\}}$, the setwise stabiliser $X_I$ acts transitively on $C$. 
\end{corollary}
\begin{proof} By definition $C$ is $(X,i)$-neighbour transitive and, by 
  Proposition \ref{ihom}, $X_\alpha$ acts transitively on the set $M^{\{i\}}$ of $i$-subsets of $M$.  
  Hence $X$ is transitive on $C\times M^{\{i\}}$, and so $X_I$ is transitive on $C$.
\end{proof}

Let $X\leq\Aut(\Gamma)$.  Then for each $i\in M$ we define an action
of $X_i=\{g\sigma\in X\,:\,i^\sigma=i\}$ on the alphabet $Q$ via the
following homomorphism:       
\[\begin{array}{c c c c}
     \varphi_i:&X_i&\longrightarrow&S_q\\ 
&(g_1,\ldots,g_m)\sigma&\longmapsto&g_i\\
\end{array}
\]  We denote the image of $X_i$ under $\varphi_i$ by $X_i^Q$.

\begin{proposition}\label{x12trans} Let $C$ be an $(X,1)$-neighbour
  transitive code in $H(m,q)$ with $\delta\geq 3$ and $|C|>1$.  Then
  $X_1^{Q}$ acts $2$-transitively on $Q$.
\end{proposition}

\begin{proof}  Let $a\in Q\backslash\{0\}$.  By replacing $C$ with
  an equivalent code if necessary, Remark \ref{remfav} allows us to
  assume that $\alpha={\bf{0}}$ and $\beta=(a^\delta,0^{m-\delta})$ are two
  codewords of $C$.  Choose any $b\in Q\backslash\{0\}$.  As
  $\delta\geq 3$ it follows that $\nu_1=(a,0^{m-1})$,
  $\nu_2=(b,0^{m-1})\in\Gamma_1(\alpha)\subseteq C_1$.  By Proposition
  \ref{ihom}, there exists $x=(g_1,\ldots,g_m)\sigma\in X_\alpha$ such
  that $\nu_1^x=\nu_2$.  Consequently, Lemma \ref{supplem} 
  implies that $1^\sigma=1$.  Thus $a^{g_1}=b$, and
  because $x\in X_\alpha$, we conclude that $g_1\in(X_1^{Q})_0$, the
  stabiliser of $0$ in $X_1^Q$.  Hence $(X_1^{Q})_0$ acts transitively on $Q\backslash\{0\}$.
  By Corollary \ref{multtrans}, $X_1$ acts transitively on $C$.  Hence
  there exists $y=(h_1,\ldots,h_m)\pi\in X_1$ such that $\alpha^y=\beta$.  As $y\in 
  X_1$ we have that $0^{h_1}=a$ and $h_1\in X_1^Q$.  Thus $X_1^Q$ acts $2$-transitively on $Q$.     
\end{proof}

\begin{corollary}\label{2transM} Let $C$ be an $(X,2)$-neighbour transitive 
  code with $|C|>1$, $X\cap\base=1$ and $\delta\geq 5$.  Then $X$ acts $2$-transitively on $M$.  
\end{corollary}

\begin{proof}  By Proposition \ref{ihom}, $X_\alpha$ acts
  $2$-homogeneously on $M$, and so $X$ has a faithful $2$-homogeneous action on $M$.  By Proposition
  \ref{x12trans}, $X_1^Q$ acts $2$-transitively on $Q$.  Thus $X_1^Q$
  has even order, and so $X$ has even order.  Therefore, by
  \cite[Lem. 2.1]{sharp}, $X$ acts $2$-transitively on $M$.        
\end{proof}

\begin{lemma}\label{restrictq} Let $C$ be an $X$-completely transitive
  code.  Then $q^m/(m+1)\leq |X|$. Moreover, if $X\lesssim S_m$ and
  $m\geq 5$, then $q \leq m-2$.
\end{lemma}
\begin{proof}
Since the diameter of $H(m,q)$ is equal to $m$, we naturally have that
$\rho\leq m$. Then, because $V(\Gamma)$ has size $q^m$ and the
distance partition of $C$ has $\rho+1$ parts, there exists $i$ such that 
$$|C_i| \geq \frac{q^m}{\rho+1} \geq \frac{q^m}{m+1}.$$  As $C_i$
is an $X$-orbit it follows that $|C_i|\leq |X|$, and so the first
inequality holds.  Now suppose $X\lesssim S_m$.  Then $|X|\leq m!$ and
so $q^m \leq (m+1)!$. If $q \geq m-1$ then 
$(m-1)^m \leq (m+1)!$, which holds if and only if $m\leq 4$.
\end{proof}

\subsection{$s$-regular and completely regular codes.}

\begin{definition}\label{cregdef}  Let $C$ be a code with covering
  radius $\rho$ and $s$ be an integer such that $0\leq s\leq\rho$.
  We say $C$ is \emph{$s$-regular} if for each vertex $\gamma\in C_i$,
  with $i\in\{0,\ldots,s\}$, and each integer $k\in\{0,\ldots,m\}$, the number
  $|\Gamma_k(\gamma)\cap C|$ depends only on $i$ and $k$.  If
  $s=\rho$ we say $C$ is \emph{completely regular}.    
\end{definition}\noindent  It follows from the definitions that any
code equivalent to an $s$-regular code is necessarily $s$-regular.  The next
three results examine the natural expectation that a completely regular code 
with large minimum distance would be small in size. 

\begin{lemma}\label{sizeofcode} Let $C$ be a code with $|C|\geq 2$ and $\delta=m$.
  Then there exists $C'$ equivalent to $C$ with
  $C'\subseteq\Rep(m,q)$.  Moreover if $C$ is $1$-regular then
  $C'=\Rep(m,q)$; if $C$ is $2$-regular and $m\geq 5$ then
  $C'=\Rep(m,2)$.
\end{lemma}

\begin{proof} Let $0,a\in Q$.  By Lemma \ref{favvertex}, there is a
  code $C'$ equivalent to $C$ which contains $\alpha=(0,\ldots,0)$ and
  $\beta=(a,\ldots,a)$, and each $\gamma\in
  C'\backslash\{\alpha,\beta\}$ at distance $\delta=m$ from $\alpha$
  is of the form $(b,\ldots,b)$ for some $b\in
  Q\backslash\{0,a\}$. As $C$ has $\delta=m$, it follows that $C'$ is a subset of the repetition code
  $\Rep(m,q)$.  We note this implies $|C|=|C'|\leq q$.    

  Assume $C$, and hence $C'$, is $1$-regular.  Suppose
  $|C'|<q$.  Then there exists $b\in Q\backslash\{0,a\}$ such that $b$ does not appear in any codeword of
  $C'$.  Let $\nu_1=(a,0,\ldots,0)$ and $\nu_2=(b,0,\ldots,0)$.  Then
  $\nu_1,\nu_2\in C'_1$.  It follows that $|\Gamma_{m-1}(\nu_1)\cap
  C'|=2$ if $m=2$, and $1$ if $m\geq 3$, while
  $|\Gamma_{m-1}(\nu_2)\cap C'|=1$ if $m=2$ and $0$ if $m\geq 3$, which
  is a contradiction.  Therefore $|C'|=q$ and $C'=\Rep(m,q)$.      

Now assume that $m\geq 5$ and $C$, and hence $C'$, is $2$-regular.  
Let $\nu_3=(a,a,0^{m-2})$.  As $\nu_3$ has weight $2$ and $m\geq 5$, we have that $\nu_3\in C'_2$.
Also, $d(\nu_3,\beta)=m-2$.  Therefore, because $C'$ is $2$-regular,
$\Gamma_{m-2}(\nu)\cap C'\neq\emptyset$ for all $\nu\in C'_2$.  Now
suppose that $q\geq 3$ and $b\in Q\backslash\{0,a\}$.  Then 
$\nu_4=(b,a,0^{m-2})\in C'_2$.  Let $\gamma=(c,\ldots,c)$, an arbitrary 
element of $C'$. 
Then \[d(\nu_4,\gamma)=\left\{\begin{array}{cl} 2&\textnormal{if
  $c=0$,}\\ m-1&\textnormal{if 
$c=a$ or $b$,}\\
m&\textnormal{if $c\in
  Q\backslash\{0,a,b\}$.}\end{array}\right.\]  Since $m\geq 5$ it
follows that $\Gamma_{m-2}(\nu_4)\cap C=\emptyset$ which is a
contradiction.  Therefore $q=2$.     
\end{proof}

\begin{lemma}\label{codesize2} Let $C$ be a $1$-regular code in $H(m,q)$ with $|C|=2$.
  Then $\delta=1$ or $m$, and $q=2$.  Moreover, if $\delta=m$ then $C$
  is equivalent to the binary repetition code $\Rep(m,2)$.     
\end{lemma}

\begin{proof}  Firstly suppose that $\delta<m$.  By Lemma
  \ref{favvertex}, $C$ is equivalent
  to $$C'=\{(0,\ldots,0),(a^\delta,0^{m-\delta})\}$$ for 
  some $a\in Q\backslash\{0\}$, and $C'$ is $1$-regular since $C$ is.
  Suppose $\delta\geq 2$, and let $\nu_1=(a,0^{m-1})$ and
  $\nu_2=(0^{\delta},a,0^{m-\delta-1})$.  Since $2\leq\delta< m$ it  
  follows that $\nu_1,\nu_2\in C'_1$.  However, we  
  observe that if $\delta\geq 3$ then $|\Gamma_{\delta-1}(\nu_1)\cap
  C'|=1$ and $|\Gamma_{\delta-1}(\nu_2)\cap C|=0$, while if $\delta=2$
  then $|\Gamma_1(\nu_1)\cap C'|=2$ and
  $|\Gamma_1(\nu_2)\cap C'|=1$, contradicting the fact that
  $C'$ is $1$-regular.  Thus $\delta=1$.  Now suppose that $q\geq 3$ and $b\in 
  Q\backslash\{0,a\}$.  Let $\nu_3=(b,0^{m-1})$ and
  $\nu_4=(0,b,0^{m-2})$.  Then $\nu_3,\nu_4\in C'_1$.  However,
  $|\Gamma_1(\nu_3)\cap C'|=2$ and $|\Gamma_1(\nu_4)\cap C'|=1$,
  contradicting the fact that $C'$ is $1$-regular.  Thus $q=2$.  Now
  suppose that $\delta=m$.  Then Lemma \ref{sizeofcode} implies that
  $C$ is equivalent to the repetition code $\Rep(m,q)$.  Since
  $|\Rep(m,q)|=q$ and $|C|=2$, it follows that $q=2$.    
\end{proof}

\begin{lemma}\label{c=2d=m} Let $C$ be a completely regular code in
  $H(m,q)$ with $m\geq 5$ and $\delta\geq 2$.  Then $|C|=2$ if and
  only if $\delta=m$.   
\end{lemma}

\begin{proof}  Suppose that $|C|=2$.  Since $m\geq 5$ it follows that
  $C$ is $1$-regular.  As $\delta\geq 2$, Lemma \ref{codesize2} implies that $\delta=m$.  Conversely suppose that $\delta=m$.
  As $C$ is completely regular, $m\geq 5$ and $\delta=m$ it follows
  that $\rho\geq 2$, and so $C$ is $2$-regular.  Therefore Lemma
  \ref{sizeofcode} implies that $C$ is equivalent to $\Rep(m,2)$.  
  Thus $|C|=2$.     
\end{proof} 

\subsection{t-designs and $q$-ary t-designs}\label{sectdes}

Let $\mathcal{D}=(\pt,\B)$ where $\pt$ is a set of \emph{points} of
cardinality $m$, and $\B$ is a set of $k$-subsets of $\pt$ called
blocks.  We say $\mathcal{D}$ is a \emph{$t-(m,k,\lambda)$ design} if
every $t$-subset of $\pt$ is contained in exactly $\lambda$ blocks of
$\B$.  We let $b$ denote the number of blocks in $\mathcal{D}$ and $r$
denote the number of blocks that contain any given point.  We say a
non-negative integer $\ell$ is \emph{block intersection number} of
$\mathcal{D}$ if there exist distinct blocks $B,B'\in\B$ such that
$|B\cap B'|=\ell$.  An \emph{automorphism of $\mathcal{D}$} 
is a permutation of $\mathcal{P}$ that preserves $\mathcal{B}$, and we let
$\Aut(\mathcal{D})$ denote the group of automorphisms of
$\mathcal{D}$.  For further concepts and definitions about $t$-designs, 
see \cite{camvan}.    

\begin{remark}\label{desrem}  Let $\pt$ be a set with cardinality $m$
  and $G\leq\sym(\pt)$.  Suppose $G$ acts $t$-homogeneously on $\pt$,
  and let $B\in \pt^{\{k\}}$.  Then $(\pt,B^G)$ forms a
  $t-(m,k,\lambda)$ design for some integer $\lambda$.  Using this
  fact, we can prove that $\PSL(2,5)$ has two orbits, $\mathcal{O}_1$,
  $\mathcal{O}_2$, on $M^{\{3\}}$ (here $M=\{1,\ldots,6\}$), each of
  which is a $2-(6,3,2)$ design, and each is the complementary design
  of the other (see \cite[Sec. 2.4 and Lem. 9.1.1]{ngthesis}).
  Also, any design with these parameters is unique up to isomorphism
  and has automorphism group isomorphic to $\PSL(2,5)$.           
\end{remark} 

For $\alpha,\beta\in V(\Gamma)$, we say $\alpha$ is \emph{covered} by $\beta$ if $\alpha_i=\beta_i$ for
each non-zero component $\alpha_i$ of $\alpha$.  Let $\mathcal{D}$ be a non-empty set of vertices of weight $k$ in
$H(m,q)$.   Then we say $\mathcal{D}$ is a \emph{$q$-ary
  $t$-$(m,k,\lambda)$ design} if for every vertex $\nu$ of weight $t$,
there exist exactly $\lambda$ vertices of $\mathcal{D}$ that cover
$\nu$.  If $q=2$, this definition coincides with the usual definition
of a $t$-design, in the sense that the set of blocks of the $t$-design
is the set of supports of vertices in $\mathcal{D}$, and as such we simply 
refer to $2$-ary $t$-designs as $t$-designs.  It is known that for a completely 
regular code $C$ in $H(m,q)$ with zero codeword and minimum distance $\delta$, the set $C(k)$ of 
codewords of weight $k$ forms a $q$-ary $t-(m,k,\lambda)$ design for some $\lambda$ with
$t=\lfloor\frac{\delta}{2}\rfloor$ \cite{ufpc1}.  Using this, we prove the following results.

\begin{lemma}\label{cgeqm1} Let $C$ be a completely regular code in
  $H(m,2)$ with $|C|\geq 2$ and $5\leq\delta<m$.  Then $|C|\geq m+1$.     
\end{lemma}

\begin{proof}  $C$ is equivalent to a completely regular
  code $C'$ that contains $\z$.  As $\delta\geq 5$, it follows that
  $C'(\delta)$ is a $2$-$(m,\delta,\lambda_2)$ design for some
  $\lambda_2$ \cite[Cor. 1.6]{camvan}.  Since $\delta<m$, 
  Fisher's inequality \cite[Thm. 1.14]{camvan} implies that
  $|C'(\delta)|\geq m$.  Consequently $|C|=|C'|\geq m+1$.        
\end{proof}

\begin{lemma}\label{nonexistcr}  There do not exist binary completely
  regular codes of length $m$ with minimum distance $\delta$ for
  $m=13$ and $\delta=5,6$, or for $m=16$ and $\delta=5,7,8$.  
\end{lemma}

\begin{proof}  Let $C$ be a binary completely regular code of length $16$
  with $\delta=5$.  By replacing $C$ with an equivalent code if
  necessary, we can assume that $\z\in C$.  Therefore $C(5)$
  forms a $2-(16,5,\lambda)$ design for some $\lambda$.  It follows that $r=15\lambda/4$, 
  and so $4$ divides $\lambda$.  There are exactly $\lambda$ codewords of
  weight $5$ whose support contains $\{1,2\}$.  Because $\delta=5$, it
  follows that the supports of any pair of these codewords intersect
  precisely in $\{1,2\}$.
  Consequently $$\lambda\leq\frac{16-2}{5-2}<5,$$ and so $\lambda=4$.
  Thus $C(5)$ forms a $2-(16,6,4)$ design and $a_5=|C(5)|=48$.  Using the fact
  that $\delta=5$, a simple counting argument gives that 
  $C(5)$ has block intersection numbers $2,1$ and $0$.  Consequently, for $\alpha\in C(5)$, 
  it holds that $\Gamma_k(\alpha)\cap C(5)\neq\emptyset$ for $k=6,8,10$, and so $C(k)\neq\emptyset$ for the same
  values of $k$.  

  Suppose that the all one vertex $\1$ is not a codeword.  Then, by \cite[Lemma 2.2]{ngalone}, 
  $C$ has covering radius $\rho\geq\delta-1=4$ and $C_\rho=\1+C$.  
  Furthermore, because $C(10)\neq\emptyset$ and $\1\in C_\rho$, it follows that $\rho\leq 6$.
  It is known that $C_\rho$ is also completely regular with distance partition $\{C_\rho,C_{\rho-1},\ldots,C_1,C\}$ \cite{neu}.  
  Thus, as $\delta=5$, it follows that $C_{\rho-i}=\1+C_i$ for $i=1, 2$ also.  Therefore if $\rho=4$ or $5$ it holds
  that $|C|(2+2\times 16+\binom{16}{2})=2^{16}$ or $|C|(2+2\times 16+2\times\binom{16}{2})=2^{16}$ respectively, 
  which is a contradiction.  Hence $\rho=6$.  As $\1\in C_6$, it follows that $k=10$ is the
  maximum weight of any codeword in $C$.  Thus the distance distribution of $C$ is equal to 
  $$a(C)=(1,0,0,0,0,48,a_6,a_7,a_8,a_9,a_{10},0,0,0,0,0,0)$$ By
  considering the MacWilliams transform of $a(C)$ (see (\ref{kracheqn})), we obtain
  the following linear constraints \cite[Lem. 5.3.3]{vanlint}: 
\begin{align*}
600-6a_7-8a_8-6a_9&\geq 0\\
-360+6a_7-8a_8+6a_9&\geq 0
\end{align*}
with $a_7,a_9\geq 0$ and $a_8>0$.  Adding these together implies that $a_8\leq 15$.  However, there exists a positive integer $\lambda'$
such that $C(8)$ forms a $2-(16,8,\lambda')$ design with
$$a_8=|C(8)|=\frac{16.15}{8.7}\lambda'=\frac{30}{7}\lambda'.$$
Thus $7$ divides $\lambda'$ and $a_8\geq 30$, which is a contradiction.  Hence $\1\in C$.
This implies that $C$ is \emph{antipodal}, that is, $\alpha+\1\in C$ for all $\alpha\in C$,
and so $a_i=a_{m-i}$ for all $i$ in $a(C)$.  Again, by applying the MacWilliams transform to $a(C)$ we
  generate twelve linear constraints that must be non-negative.  However, it is straight forward to
  obtain a contradiction from these constraints (see \cite[Lem. 7.4.2.2]{ngthesis}).  Thus no such code exists with 
  $m=16$ and $\delta=5$.  For the other values of $m$, $\delta$, we 
  follow a similar argument to that given in \cite[Lem. 6]{nonexist2}
  to prove that binary completely regular codes with these parameters 
  do not exist (see \cite[Lem. 7.4.2.1]{ngthesis}).
\end{proof}

\section{Basic Cases}\label{secamsm}

We now begin to prove Theorem \ref{maink=1}.  We first consider the
case $X\cong A_m$ or $S_m$, and then the case $\delta=m$.  

\begin{remark} If $C$ is an $X$-completely transitive code, then $C$ is 
  completely regular \cite{giupra}.  Furthermore, if $\delta\geq 5$ then $C$ has covering radius 
  $\rho\geq 2$.  Thus $C$ is at least $(X,2)$-neighbour transitive, so by Proposition
  \ref{ihom}, $X_\alpha$ acts $2$-homogeneously on $M$.  As we only consider completely
  transitive codes with $\delta\geq 5$ for the remainder of this paper, 
  from now on we use both these results without further reference.      
\end{remark}

\begin{proposition}\label{xcongsm} Let $C$ be an $X$-completely transitive code in
  $H(m,q)$ with $|C|>1$, $X\cap\base=1$, $X \cong A_m$ or $S_m$ and $\delta \geq
  5$. Then $q=2$, $X \cong S_m$, $X_{\alpha} \cong A_m$ and $C$ is
  equivalent to $\Rep(m,2)$.
\end{proposition}

\begin{proof} As $m\geq\delta\geq 5$ the code $C$ is at least
  $2$-regular.  If $m=5$ then by Lemma \ref{sizeofcode}, $q=2$ and 
  $C$ is equivalent to $\Rep(m,2)$.  In this case, since $X$ is
  transitive on $C$, $X_\alpha$ has index $2$ and hence is normal in
  $X$.  Thus $X\cong S_5$ and $X_\alpha\cong A_5$.  Thus we may assume
  that $m\geq 6$.  Since $X\cong A_m$ or $S_m$, it follows that (the
  stabiliser of the first entry) $X_1 \cong A_{m-1}$ or $S_{m-1}$.  By
  Proposition \ref{x12trans}, $X_1$ has a $2$-transitive action of
  degree $q$.  By Lemma \ref{restrictq}, for $m\geq\delta\geq 5$, we
  have that $q\leq m-2$.  Thus, by considering the $2$-transitive
  actions of $A_n$ and $S_n$ for an arbitrary $n$ \cite{cambull}, we
  have, since $m\geq 6$, that $X \cong S_m$ and $q=2$.      

  Now consider the group $X_\alpha$, and suppose first that $A_m$ is
  not a subgroup of $X_\alpha$.  As $q=2$ it follows that
  $|X:X_{\alpha}|=|C| \leq 2^m$, and since $X_{\alpha}$ acts
  $2$-homogeneously and hence primitively on $M$, a result by
  Mar{\'o}ti \cite{maroti} gives us that $|X_{\alpha}|\leq 3^m$. 
  It follows that $m!/2^m=|X|/2^m\leq |X_\alpha|\leq 3^m$.  Thus $m!\leq
  6^m$, which implies that $m\leq 13$.  By the Sphere Packing Bound
  \cite[Thm. 5.2.7]{vanlint}, 
  $|C|(1+m+{m \choose 2}) \leq 2^m$, and so $|X_{\alpha}| \geq
  m!(1+m+{m \choose 2})/2^m$.  Now, from \cite{camperm} and
  \cite{kantor}, the only $2$-homogeneous groups with degree $m\leq 13$ 
  that are not $A_m$ or  $S_m$ are the projective groups, the
  affine groups, $M_{11}$ with degree $11$ or $12$, and $M_{12}$ with
  degree $12$.  However, we see in each case that the orders of these
  groups are always less than $m!(1+m+{m \choose 2})/2^m$, which is a
  contradiction.  Thus $A_m$ is a subgroup of $X_\alpha$.  Since
  $|C|>1$, it follows that $X_\alpha\cong A_m$ and $|C|=2$.  Therefore, by
  Lemmas \ref{codesize2} and \ref{c=2d=m}, $C$ is equivalent to
  $\Rep(m,2)$.        
\end{proof}

\begin{proposition}\label{d=m} Let $C$ be an $X$-completely transitive code with
  $m\geq 5$, $|C|\geq 2$, $X\cap\base=1$ and $\delta=m$.  Then $C$
  is equivalent to the repetition code $\Rep(m,2)$, $X\cong S_m$ and
  $X_\alpha\cong A_m$.        
\end{proposition}

\begin{proof}  As $C$ is completely regular with $\delta=m\geq 5$,
  it follows that $\rho\geq 2$ and so $C$ is at least $2$-regular.
  Thus, by Lemma \ref{sizeofcode}, $C$ is equivalent to the repetition code
  $\Rep(m,2)$, so we just need to prove the statement about the groups $X$ and $X_\alpha$.  
  By replacing $C$ with an equivalent code if 
  necessary, let us assume that $C=\Rep(m,2)$.  As $|C|=2$
  we have that $|X:X_\alpha|=2$.  Furthermore, by Corollary
  \ref{multtrans}, $X_1$ acts transitively on $C$, and so
  $|X_1:X_{1,\alpha}|=2$.  \emph{We claim that $A_m\lesssim X$, from which,
  by Proposition \ref{xcongsm}, we obtain $X\cong S_m$ and
  $X_\alpha\cong A_m$.}  We repeatedly use the classification of
  $2$-transitive groups to prove this claim (see \cite{camperm}).

  Suppose to the contrary that $A_m\not\lesssim X$.  By Proposition \ref{ihom},
  $X_\alpha$ (and so $X$ also) is $i$-homogeneous on $M$ for all
  $i\leq \lfloor\frac{\delta-1}{2}\rfloor=\lfloor\frac{m-1}{2}\rfloor$, and
  note that any $i$-homogeneous group is also $(m-i)$-homogeneous.  
  By the classification of $2$-transitive groups, $X$ is not
  $6$-transitive (see \cite[Sec. 7.3]{dixmort}), and hence is not
  $6$-homogeneous by \cite{kantor}.  Thus $m\leq 12$ and if $m$ is odd
  then $m=5$, or $m=9$ with $\PGL(2,8)\leq X_\alpha<
  X\leq\PGaL(2,8)$.  However in the latter case $|X:X_\alpha|=1$ or
  $3$, which is a contradiction.  Also, if $m=5$ then, by
  \cite[Thm. 9.4B]{dixmort}, $X\leq Z_5.Z_4$, since
  $A_5\not\lesssim X$, and so $X_\alpha\lesssim D_{10}$, which is not
  $2$-homogeneous, a contradiction.  Thus $m\in \{6,8,10,12\}$ and
  $X$, $X_\alpha$ are $(\frac{m-2}{2})$-homogeneous on $M$.        

  If $m=12$ then $X$, $X_\alpha$ are $5$-transitive by \cite{kantor}, and the only possibility is
  $X\cong M_{12}$, which has no index $2$-subgroup $X_\alpha$.
  Similarly, if $m=10$ then $X$, $X_\alpha$ are $4$-transitive by
  \cite{kantor}, 
  but the only $4$-transitive subgroups of $S_{10}$ are $A_{10}$ and
  $S_{10}$.  Next suppose $m=8$.  In this case $C=\Rep(m,8)$, which
  has covering radius $\rho=4$.  The only $3$-homogeneous subgroup
  $X$ of $S_8$, not containing $A_8$, with a subgroup of index $2$ is
  $X\cong\PGL(2,7)$, with $X_\alpha\cong\PSL(2,7)$.  However, since $C$ is
  $X$-completely transitive, $X$ is transitive on $C_4$, the set of
  $\binom{8}{4}=70$ vertices of weight $4$.  This is impossible since
  $|X|$ is not divisible by $5$.  Thus $m=6$.  

  In this final case, $C=\Rep(2,6)$, which has covering radius
  $\rho=3$, and $C_3$ consists of the $20$ weight $3$ vertices in
  $H(6,2)$.  The only $2$-homogeneous subgroup $X$ of $S_6$, not
  containing $A_6$, with an index $2$ subgroup is $X\cong\PGL(2,5)$, with
  $X_\alpha\cong\PSL(2,5)$.  We note that because $q=2$, it follows
  that $X_\alpha=X\cap \topg\leq\topg$, where $\alpha={\bf{0}}$. Let
  $H=N_{\topg}(X_\alpha)\cong\PGL(2,5)$.  Note also that if $g=(h,\ldots,h)\in\base$ 
  for $1\neq h\in S_2$, then $X\leq\Aut(C)=\langle g,\topg\rangle$.  
  Suppose that $x=g\sigma\in X$ with 
  $\sigma\in X_\alpha$.  Then $x\sigma^{-1}=g\in X$, and so
  $X\cap\base$ is a non-trivial normal $2$-subgroup.  However this
  contradicts the fact $X\cong\PGL(2,5)$.  Therefore we deduce that
  $X=X_\alpha\cup g(H\backslash X_\alpha)$.  By Remark
  \ref{desrem}, the induced action of $X_\alpha$ on $M^{\{3\}}$, the
  set of $3$-subsets of $M$, has two orbits, $\mathcal{O}_1$,  
  $\mathcal{O}_2$.  Moreover, each orbit forms a $2-(6,3,2)$ design
  and is the complementary design of the other.  Also, because $X$
  acts transitively in its induced action on $M^{\{3\}}$, and because  
  $\PSL(2,5)\unlhd \PGL(2,5)$ we have that
  $\Delta=\{\mathcal{O}_1,\mathcal{O}_2\}$ is a system of
  imprimitivity for the action of $X$ on $M^{\{3\}}$.  Let
  $C(\mathcal{O}_i)$ be the set of vertices in $H(6,2)$ whose supports
  are the elements of $\mathcal{O}_i$ for each $i$, so
  $C_3=C(\mathcal{O}_1)\cup C(\mathcal{O}_2)$.  If $x\in X_\alpha$ it
  follows that $C(\mathcal{O}_1)^x=C(\mathcal{O}_1)$.  If $x\in
  X\backslash X_\alpha$ then $x=g\sigma$ with $\sigma\in H\backslash
  X_\alpha$.  It follows that
  $C(\mathcal{O}_2)^\sigma=C(\mathcal{O}_1)$, and because
  $\mathcal{O}_2$ is the complementary design of $\mathcal{O}_1$,
  $C(\mathcal{O}_1)^g=C(\mathcal{O}_2)$.  Thus
  $C(\mathcal{O}_1)^x=C(\mathcal{O}_1)$.  Consequently, $C_3$ is not
  an $X$-orbit, which is a contradiction.  Thus the claim is proved.    
\end{proof} 

\section{New Hypothesis}

By Lemma \ref{c=2d=m}, if $C$ is completely regular in
$H(m,q)$ with $m\geq 5$ and $\delta\geq 2$ then $|C|=2$ if and only if
$\delta=m$.  Therefore, given Propositions \ref{xcongsm} and \ref{d=m},
and Corollary \ref{2transM},  to complete the proof of
Theorem \ref{maink=1}, we only need to consider $X$-completely
transitive codes with $\delta<m$ (which is equivalent to $|C|>2$) such 
that $X\cap\base=1$, and $X$ is a $2$-transitive subgroup of $S_m$ not
containing $A_m$.  We bring this together in the following hypothesis.           

\begin{hypothesis}\label{mainhyp} Let $C$ be an $X$-completely
  transitive code in $H(m,q)$ with $|C|>2$, minimum distance $\delta$ 
  satisfying $5\leq\delta<m$, and $X\cap\base=1$ such that $\mu(X)\cong X$ 
  is $2$-transitive not containing $A_m$.   
\end{hypothesis}    
 
\begin{lemma}\label{boundq} Let $C$ be an $X$-completely transitive
  code that satisfies Hypothesis \ref{mainhyp}.  Then either (i) $q=2$, 
  $X_1^Q=S_2$ and $\delta\leq m/2$, or (ii) $q=3$, $X_1^Q=S_3$ and $8\leq m\leq 24$.
  Moreover, $X_1$ is not perfect. 
\end{lemma}

\begin{proof} Since $X$ acts $2$-transitively, it acts primitively on
  $M$, and because it does not contain $A_m$ we have that $|X|\leq
  3^m$ for $m\leq 24$, and $|X|\leq 2^m$ otherwise \cite{maroti}.  By
  Lemma \ref{restrictq}, $q^m/(m+1)\leq |X|$ from which we deduce that either $q=2$;
  $q=3$ and $m\leq 24$; or $q=4$ and $m\leq 7$.  The only binary completely regular code with $m/2<\delta<m$
  has minimum distance $4$ \cite{ngalone}.  Therefore, because $\delta\geq 5$, if $q=2$
  it follows that $\delta\leq m/2$, which also implies that $m\geq 10$.  Suppose now that $q\in\{3,4\}$.  If $m=7$, then 
  the only $2$-transitive groups $X$ (not
  containing $A_7$) are $X\cong\PSL(3,2)$ and $\AGL(1,7)$, so $|X|\leq
  168<3^7/8$, a contradiction.  If $m=6$ then
  $X\cong\PSL(2,5)$ or $\PGL(2,5)$, so $q^6/7\leq |X|\leq 120$, which
  implies that $q=3$ and $X=\PSL(2,5)$.  However this implies that
  $X_1\cong D_{10}$, which does not act as $S_3$ on $Q$, contradicting
  Proposition \ref{x12trans}.  Since $m\geq 6$, we deduce that
  $q=3$ and $8\leq m\leq 24$.  The claims about $X_1^Q$ follow from
  Proposition \ref{x12trans}.  It follows that $X_1^Q$ is soluble and,
  in particular, $X_1$ is not perfect.  
\end{proof}

\subsection{$X$ is $2$-transitive of Affine Type}\label{affinetype}

Let $C$ be a code that satisfies Hypothesis \ref{mainhyp}.  The group
$X$ acts faithfully and $2$-transitively on $M$ and so $X$ is either
of affine or almost simple type.  We consider the affine case first.         

\begin{proposition}\label{notaffine} There are no $X$-completely
  transitive codes in $H(m,q)$ satisfying Hypothesis \ref{mainhyp}
  such that $X$ is of affine type. 
\end{proposition}

\begin{proof}  Throughout this proof we repeatedly use the
  classification of $2$-transitive groups (see \cite{camperm}).
  Suppose $C$ is an $X$-completely transitive code satisfying
  Hypothesis \ref{mainhyp} such that $X$ is of affine type.  Then
  $X=NX_1\lesssim\AGL(n,r)$ for some $n,r$ with $r$ a prime and
  $m=r^n$, and with $N$ the unique minimal normal subgroup of $X$
  of order $r^n$.  Recall also, by Lemma \ref{boundq}, 
  that either $q=2$ and $\delta\leq m/2$ (so $m\geq 10$), or $q=3$ and $8\leq m\leq 24$.  We deduce from
  Lemma \ref{restrictq}
  that \begin{equation}\label{mainafeq1}q^{r^n}\leq|X|(r^n+1)\leq|\AGL(n,r)|(r^n+1)\leq   
    r^{n^2+2n},\end{equation} and
  so \begin{equation}\label{mainafeq2}f_n(r):=\frac{r^n}{\log(r)}\leq\frac{n^2+2n}{\log(q)}.\end{equation}    
We claim that $r,n,q$ are as in one of the rows in Table
\ref{afftable}.  Suppose first that $r=2$.  If $q=3$ then
(\ref{mainafeq2}) implies that $2^n\leq (n^2+2n)(\log(2)/\log(3))$, and so
$n\leq 3$.  Furthermore, because $m=r^n\in[8,24]$ when $q=3$ it
follows that $n=3$ as in row $1$.  If $q=2$ then (\ref{mainafeq2})
implies that $10\leq 2^n<n^2+2n$, and so $n=4$ or $5$ as in row $2$. 
Suppose now that $r\geq 3$.  In this case, $f_n(r)$ is an increasing function for a fixed $n$.  
Thus (\ref{mainafeq2}) implies that \begin{equation}\label{mainafeq3}3^n\leq \frac{(n^2+2n)\log(3)}{\log(q)}\end{equation}
If $q=3$ we deduce that $n=1$.  Hence $f_1(r)\leq 3/\log(3)$, and so $r=3$ and $m=3$, which 
is a contradiction.  Thus $q=2$ with $m\geq 10$, and (\ref{mainafeq3}) implies that $n\leq 2$.  If $n=2$
then $f_2(r)\leq 8/\log(2)$, which holds only if $r=3$ (recall $r$ is a
prime), and so $m=9$, contradicting the fact $m\geq 10$.  Thus $n=1$.  Consequently
$f_1(r)\leq 3/\log(2)$, which holds only if $m=r\leq 9$, again a contradiction.  Thus the claim holds.    
\begin{table}[t]
\centering
\begin{tabular}{l|c|c|c}
\hline\hline
Row&$r$&$n$&$q$\\
\hline
1&2&3&3\\
2&2&4 or 5&2\\\hline
\end{tabular}
\caption{Possible $r,n,q$ in Affine case}
\label{afftable}
\end{table} 

Consider row $1$, so $X$ is a $2$-transitive subgroup of
$\AGL(3,2)$, and by Proposition \ref{x12trans},
$|X_1|=|X\cap\GL(3,2)|$ is even.  It follows that $X\cong \AGL(3,2)$,
but then $X_1\cong\GL(3,2)$ is perfect, contradicting Lemma
\ref{boundq}.  In row $2$, 
$m=16$ or $32$. Suppose that $m=32$.  Then
$X\lesssim \AGL(5,2)$, and as before $|X_1|$ is even.  This means that
$X\not\lesssim\GaL(1,32)$ (of order $31.5$), and hence $X_1\cong
\GL(5,2)$.  However, in this case $X_1$ is perfect contradicting
Lemma \ref{boundq}.  Thus $m=16$ and $X_1\lesssim \GL(4,2)$.  By Lemma \ref{boundq}, 
$\delta\leq 8$, and by Lemma \ref{nonexistcr} there do not exist
binary completely regular codes of length $16$ with $\delta=5,7$ or
$8$.  Thus $\delta=6$.  Any completely regular code in $H(16,2)$ 
with $\delta=6$ is equivalent to the Nordstrom-Robinson code \cite{nord}, 
which consists of $256$ codewords.  Thus $|C|=256$.  Furthermore, by Corollary
\ref{multtrans}, $X_1$ acts transitively on $C$, and so $256$ divides
$|\GL(4,2)|$, which is a contradiction.
\end{proof}

\subsection{$X$ is $2$-transitive of Almost Simple Type}\label{sk=1blkfaith}

In this section we consider codes that satisfy Hypothesis
\ref{mainhyp} such that $X$ is of almost simple type.  The group
$\Aut(\Gamma)$ has a natural action on $\Omega=Q\times M$ where
$h\sigma\in \Aut(\Gamma)$ maps $(a,i)$ to $(a^{h_i},i^\sigma)$.  It is
a consequence of Proposition \ref{x12trans}, and the fact that $X$
induces the $2$-transitive group $\mu(X)$ on $M$, that $X$ acts
transitively on $\Omega$.  In this action,
$\mathcal{B}=\{Q\times\{i\}\,:\,i \in M\}$ is a system 
of imprimitivity and $X_B=X_1$ where $B=Q\times\{1\}$, so $X_1^Q$ is
permutationally isomorphic to $X_B^B$.  Furthermore, it is a
consequence of a result by the third author with Schneider that there
exists $g\in \Aut(\Gamma)$ such that $X^g\leq X_1^Q\wr\mu(X)$
\cite{cherylembedding}.  The group $X^g$ is of almost simple type,
satisfies Hypothesis 1 of \cite{aliceblock}, and is faithful on
$\mathcal{B}$.  All groups with these properties are classified in
\cite[Thm. 1.4]{aliceblock}, and so the possibilities for $X$, $m$,
$q=|B|$ are listed in \cite[Tables 2 and 3]{aliceblock}.  However,
recall from Lemma \ref{boundq} that either $q=3$ and $8\leq m\leq 24$,
or $q=2$.  The only possibilities in \cite[Tables 2 and 3]{aliceblock}
that have $q=3$ and $8\leq m\leq 24$ are $\PSL(n,r)\lesssim X\lesssim
\PGaL(n,r)$ with $m=(r^n-1)/(r-1)$ for $(n,r)=(2,16),(3,3)$ or
$(3,4)$.  In each case $3^m/(m+1)>|X|$, contradicting Lemma
\ref{restrictq}.  Thus $q=2$, and the cases for which this holds in
\cite[Tables 2 and 3]{aliceblock}, excluding the Symmetric group case,
are as in Table \ref{blockfaith}.        
\begin{table}[t]
\centering
\begin{tabular}{l l c c}
\hline\hline
Line&$X$&$m$&Conditions\\
\hline
1&$\PGaL(2,8)$&$28$\\
2&$HS$&$176$\\
3&$Co_3$&276\\
4&$M_{11}$&11\\
5&$M_{22}\rtimes C_2$&22\\\hline
6&$\Sp(2\ell,2)$&$2^{2\ell-1}-2^{\ell-1}$&$\ell\geq 3$\\
7&$\Sp(2\ell,2)'$&$2^{2\ell-1}+2^{\ell-1}$&$\ell\geq 3$\\\hline
8&$\unrhd \Ree(r)$&$r^3+1$&$r=3^f$, $f\geq 3$ and odd\\\hline
9&$\unrhd\PSU(3,r)$&$r^3+1$&$r\geq 3$\\\hline
10&$\unrhd\PSL(n,r)$&$\frac{r^n-1}{r-1}$&$n\geq 2$ and $(n,r)\neq (2,2),\,(2,3)$\\
\hline
\end{tabular}
\caption{Possible $X$ and $m$ in Almost Simple Case}
\label{blockfaith}
\end{table}
\begin{proposition}\label{notas}  There are no $X$-completely
  transitive codes in $H(m,q)$ satisfying Hypothesis \ref{mainhyp}
  such that $X$ is of almost simple type.
\end{proposition}

\begin{proof}  Throughout this proof, $C$ is an $X$-completely
  transitive code in $H(m,q)$ that satisfies Hypothesis \ref{mainhyp}
  such that $X$ is of almost simple type.  From our discussion above,
  $q=2$ and $X$, $m$ are as in one of the lines of Table
  \ref{blockfaith}.  Moreover, by Lemma \ref{boundq}, $\delta\leq m/2$ and $m\geq 10$.  
  We now consider each of the lines of Table \ref{blockfaith}, repeatedly using the classification of
  $2$-transitive groups (see \cite{camperm}).    

\underline{{\bf{Lines $1-3$:}}} In each case, $2^m/(m+1)>|X|,$
contradicting Lemma \ref{restrictq}, and so no such code 
exists.   

\underline{{\bf{Line $4$:}}}  In this case $X\cong M_{11}$ and
$m=11$.  As $\delta\leq m/2$ it follows that $\delta=5$.  
By the main result of \cite{hadpap}, $C$ is equivalent to 
the punctured Hadamard $12$ code, and so $|C|=24$.  As $X$ acts
transitively on $C$ we have that $X_\alpha$ is a subgroup of index
$24$ in $M_{11}$, and hence $X_\alpha$ is a subgroup of index $2$ in a
maximal subgroup isomorphic to $\PSL(2,11)$ (see \cite{atlas}).  However 
this contradicts the fact that $\PSL(2,11)$ is simple.     

\underline{{\bf{Line $5$:}}}  In this case $X\cong M_{22}\rtimes C_2$,
$m=22$, $X_\alpha$ is $2$-homogeneous of degree $22$ and therefore
$2$-transitive \cite{kantor}.  However, the only $2$-transitive proper
subgroup of $X$ is $M_{22}$, so $|C|\leq 2$, which is a contradiction.

\underline{{\bf{Lines $6-7$:}}}  In this case $m=2^{2{\ell}-1}\pm 2^{\ell-1}$ with $\ell\geq 3$ 
and $|X|<2^{(\ell^2+\ell)/2}$ \cite[Table 4]{akospraeg}.  However,
for $\ell\geq 3$, it holds that $m+1\leq 2^{2{\ell}-1}+2^{\ell-1}+1<2^{2\ell}$ and 
$$m-2\ell\geq 2^{2{\ell}-1}- 2^{\ell-1}-2\ell\geq 2^{2\ell-2}\geq\frac{\ell^2+\ell}{2}.$$ By Lemma \ref{restrictq}, $|X|\geq
2^m/(m+1)>2^{m-2\ell}\geq 2^{(\ell^2+\ell)/2}$, which is a
contradiction.  

\underline{{\bf{Lines $8-9$:}}}  Here $T\leq X\leq \Aut(T)$ with
$T\cong \PSU(3,r)$ or $\Ree(r)$, and $r=p^f\geq 3$ for a prime $p$ and
positive integer $f$.  In both cases $|X|\leq (r^3+1)r^3(r^2-1)f\leq
2r^{12}/(r^3+2)$.  By Lemma \ref{restrictq}, $|X|\geq
2^{r^3+1}/(r^3+2)$ and hence $r^3\log(2)\leq 12\log(r)$.  The
expression $x^3/\log(x)$ is an increasing function in $x$ for $x\geq  
e^{\frac{1}{3}}$.  As $3^3/\log(3)>12/\log(2)$ it follows that
$r^3/\log(r)>12/\log(2)$, which is a contradiction.  

\underline{{\bf{Lines $10$:}}}  Here $\PSL(n,r)\lesssim
X\lesssim\PGaL(n,r)$ with $r=p^f$ for a prime $p$ and
$m=(r^n-1)/(r-1)<r^n$.  By applying Lemma \ref{restrictq} we observe
that \begin{equation}\label{aseqn1}2^m/r^n\leq 2^m/(m+1)\leq |X|\leq
  |\PGaL(n,r)|\leq r^{n^2},\end{equation} and
so \begin{equation}\label{aseqn}g_n(r):=\frac{\frac{r^n-1}{r-1}}{\log(r)}\leq\frac{n^2+n}{\log(2)}.\end{equation} 
By first considering the case $r=2$, we deduce from
(\ref{aseqn}) that $n\leq 4$.  Now, for a fixed $n$, the
function $g_n(r)$ is increasing for $r\geq 3$.  Thus $g_n(3)\leq
(n^2+n)/\log(2)$, from which we deduce that $n\leq 3$.  
By letting $n=2$ or $3$ in (\ref{aseqn1}) and using $|\PGaL(n,r)|$ as
an upper bound, we find that $r\leq 16$ or $4$ respectively.  Recalling
that $m\geq 10$, it follows that $r,n$ are as in one of the columns of Table \ref{parameters}.  

Consider column $1$, so $X\cong\PSL(4,2)$.  In this case
$X_1\cong \AGL(3,2)$ is perfect contradicting Lemma \ref{boundq}.
Now consider column $2$, so $n=3$ and
$r\in\{3,4\}$. Consequently, 
$m=13$ or $21$.  As $\PGaL(3,r)$ is not $3$-transitive, it follows
that $X$ is not $3$-transitive, and therefore, by \cite{kantor}, is
not $3$-homogeneous.  Thus Proposition \ref{ihom} implies that
$\delta\leq 6$.  By Lemma \ref{nonexistcr}, binary completely regular
codes with these parameters for $m=13$ do not exist. Therefore
$(r,m)=(4,21)$.  Since $21$ is not a prime power it follows that
$X_\alpha$ is a $2$-transitive almost simple subgroup of $X$ and
therefore $X_\alpha$ contains $\PSL(3,4)$.  Hence $|C|=|X:X_\alpha|\leq 6$.  
However, Lemma \ref{cgeqm1} implies that $|C|\geq m+1=22$. 
Thus column $2$ does not hold.\begin{table}[t]
\centering
\begin{tabular}{l|c|c|c}
\hline\hline
Column&1&2&3\\
\hline
$r$&2&3 or 4&$\leq 16$\\
$n$&4&3&2\\
\hline
\end{tabular}
\caption{Possible $r,n$ in $\PSL(n,r)$ case}
\label{parameters}
\end{table} 
In column $3$, $n=2$ with $r\leq 16$ and $m=r+1$, and because $m\geq 10$, it follows
that $r=9, 11, 13$ or $16$.  Since $X_\alpha$ is
$2$-homogeneous, we deduce in each case that $X_\alpha$ is
$2$-transitive of degree $r+1$ \cite{kantor}.  For these values of
$r$, every $2$-transitive subgroup of degree $r+1$ of $\PGaL(2,r)$
contains $\PSL(2,r)$, and so $\PSL(2,r)\lesssim X_\alpha\leq
X\lesssim\PGaL(2,r)$.  Hence $|C|=|X:X_\alpha|$ divides 
$|X:\PSL(2,r)|$ which divides $4, 2, 2, 4$ for $r=9,11,13,16$
respectively.  However, Lemma \ref{cgeqm1} implies that $|C|\geq
m+1=r+2$, which is a contradiction in each case.      
\end{proof}

\section{Proof of Theorem \ref{maink=1}}\label{proof}  

Let $C$ be an $X$-completely transitive code in $H(m,q)$ with
$\delta\geq 5$ and $X\cap\base=1$.  Firstly suppose that $X$ does not contain $A_m$.
Furthermore, suppose that $|C|>2$, so $C$ satisfies Hypothesis
\ref{mainhyp}.  By Corollary \ref{2transM}, $X$ is $2$-transitive, so
$X$ is either of affine or almost simple type.  However, it follows from Propositions
\ref{notaffine} and \ref{notas} that no such code exists.
Thus $|C|=2$, which by Lemma \ref{c=2d=m} holds if and only if
$\delta=m$.  Therefore, by Proposition \ref{d=m}, $X\cong S_m$ which
is a contradiction.  Therefore $A_m\lesssim X$.  Consequently
Proposition \ref{xcongsm} implies that $C$ is equivalent to the binary
repetition code $\Rep(m,2)$, and that $X\cong S_m$ and $X_\alpha\cong
A_m$.  

Conversely suppose $C$ is equivalent to $\Rep(m,2)$ with $m\geq
5$.  We saw in Example \ref{exk=1} that $\Rep(m,2)$ is $X$-completely 
transitive with $X\cap\base=1$, $X\cong S_m$ and $X_\alpha\cong A_m$.
As $C$ is equivalent to $\Rep(m,2)$ there exists $y\in\Aut(\Gamma)$
such that $\Rep(m,2)^y=C$, and therefore $C$ has minimum distance
$\delta=m$.  Moreover, by Remark \ref{remfav}, $C$ is $X^y$-completely
transitive. Since $X^y\cap\base=1$ if and only if
$X\cap\base=1$, we have that $|C|\geq 2$, $X^y\cap\base=1$ and
$m=\delta\geq 5$ satisfying the required conditions of Theorem
\ref{maink=1}.


\end{document}